\documentclass[12pt]{article}
\usepackage{a4wide}
\usepackage{amsmath,amsfonts,amssymb,amsthm}
\usepackage{epsfig}

\usepackage{color}
\usepackage{enumerate}

\usepackage{hyperref}
\usepackage{breakurl}

\definecolor{modra3}{rgb}{.1,.0,.4}
\hypersetup{colorlinks=true, linkcolor=modra3, urlcolor=modra3, citecolor=modra3, pdfpagemode=UseNone, pdfstartview=} 

\theoremstyle{plain}

\newtheorem{theorem}{Theorem} 
\newtheorem{lemma}[theorem]{Lemma}
\newtheorem{proposition}[theorem]{Proposition}
\newtheorem{observation}[theorem]{Observation}
\newtheorem{corollary}[theorem]{Corollary}
\newtheorem{conjecture}[theorem]{Conjecture}
\newtheorem{fact}[theorem]{Fact}

\newtheorem{problem}{Problem}

\begin{document}

\title{The $\mathbb{Z}_2$-genus of Kuratowski minors\thanks{The research was partially performed during the BIRS workshop ``Geometric and Structural Graph Theory'' (17w5154) in August 2017 and during a workshop on topological combinatorics organized by Arnaud de Mesmay and Xavier Goaoc in September 2017.}}

\author{
Radoslav Fulek\thanks{IST, Klosterneuburg, Austria;
\texttt{radoslav.fulek@gmail.com}. Supported by Austrian Science Fund (FWF): M2281-N35}
  \and
Jan Kyn\v{c}l\thanks{Department of Applied Mathematics, 
Charles University, Faculty of Mathematics and Physics, 
Malostransk\'e n\'am.~25, 118 00~ Praha 1, Czech Republic;
\texttt{kyncl@kam.mff.cuni.cz}. Supported by project 
19-04113Y of the Czech Science Foundation (GA\v{C}R), by the Czech-French collaboration
project EMBEDS II (CZ: 7AMB17FR029, FR: 38087RM) and by Charles University project UNCE/SCI/004.}
} 

\date{}

\maketitle


\begin{abstract}
A drawing of a graph on a surface is \emph{independently even} if every pair of nonadjacent edges in the drawing crosses an even number of times. The \emph{$\mathbb{Z}_2$-genus} of a graph $G$ is the minimum $g$ such that $G$ has an independently even drawing on the orientable surface of genus $g$.  An unpublished result by Robertson and Seymour implies that for every $t$, every graph of sufficiently large genus contains as a minor a projective $t\times t$ grid or one of the following so-called \emph{$t$-Kuratowski graphs}: $K_{3,t}$, or $t$ copies of $K_5$ or $K_{3,3}$ sharing at most two common vertices. We show that the $\mathbb{Z}_2$-genus of graphs in these families is unbounded in $t$; in fact, equal to their genus. 
Together, this implies that the genus of a graph is bounded from above by a function of its $\mathbb{Z}_2$-genus, solving a problem posed by Schaefer and \v{S}tefankovi\v{c}, and giving an approximate version of the Hanani--Tutte theorem on orientable surfaces. 
We also obtain an analogous result for Euler genus and Euler $\mathbb{Z}_2$-genus of graphs.
\end{abstract}


\section{Introduction}
\label{section_intro}

The \emph{genus} $\mathrm{g}(G)$ of a graph $G$ is the minimum $g$ such that $G$ has an embedding on the orientable surface $M_g$ of genus $g$.
We say that two edges in a graph are \emph{independent} (also \emph{nonadjacent}) if they do not share a vertex. The \emph{$\mathbb{Z}_2$-genus} $\mathrm{g}_0(G)$ of a graph $G$ is the minimum $g$ such that $G$ has a drawing on $M_g$ with every pair of independent edges crossing an even number of times. Clearly, every graph $G$ satisfies $\mathrm{g}_0(G)\le \mathrm{g}(G)$.

The Hanani--Tutte theorem~\cite{Ha34_uber,Tutte70_toward} states that $\mathrm{g}_0(G)=0$ implies $\mathrm{g}(G)=0$. The theorem is usually stated in the following form, with the optional adjective ``strong''.

\begin{theorem}[The (strong) Hanani--Tutte theorem{~\cite{Ha34_uber,Tutte70_toward}}]
\label{theorem_strong}
A graph is planar if it can be drawn in the plane so that no pair of independent edges crosses an odd number of times.
\end{theorem}

Theorem~\ref{theorem_strong} gives an interesting algebraic characterization of planar graphs that can be used to construct a simple polynomial algorithm for planarity testing~\cite[Section 1.4.2]{Sch13_hananitutte}. 

Pelsmajer, Schaefer and Stasi~\cite{PSS09_pp} extended the strong Hanani--Tutte theorem to the projective plane, using the list of minimal forbidden minors. Colin de Verdi{\`e}re et al.~\cite{CKPPT16_direct} recently provided an alternative proof, which does not rely on the list of forbidden minors.

\begin{theorem}[The (strong) Hanani--Tutte theorem on the projective plane{~\cite{CKPPT16_direct,PSS09_pp}}]
\label{theorem_pp_strong}
If a graph $G$ has a drawing on the projective plane such that every pair of independent edges crosses an even number of times, then $G$ has an embedding on the projective plane.
\end{theorem}

Whether the strong Hanani--Tutte theorem can be extended to some other surface than the plane or the projective plane has been an open problem. Schaefer and {\v{S}}tefankovi{\v{c}}~\cite{SS13_block} conjectured that $\mathrm{g}_0(G)=\mathrm{g}(G)$ for every graph $G$ and showed that a minimal counterexample to the extension of the strong Hanani--Tutte theorem on any surface must be $2$-connected. Recently, we have found a counterexample on the orientable surface of genus $4$~\cite{FK17_genus4}.

\begin{theorem}[{\cite{FK17_genus4}}]
\label{theorem_counter_genus4}
There is a graph $G$ with $\mathrm{g}(G)=5$ and $\mathrm{g}_0(G)\le 4$. Consequently, for every positive integer $k$ there is a graph $G$ with $\mathrm{g}(G)=5k$ and  $\mathrm{g}_0(G)\le 4k$.
\end{theorem}

The \emph{Euler genus} $\mathrm{eg}(G)$ of $G$ is the minimum $g$ such that $G$ has an embedding on a surface of Euler genus $g$. The \emph{Euler $\mathbb{Z}_2$-genus} $\mathrm{eg}_0(G)$ of $G$ is the minimum $g$ such that $G$ has an independently even drawing on a surface of Euler genus $g$.

Schaefer and {\v{S}}tefankovi{\v{c}}~\cite{SS13_block} conjectured that $\mathrm{eg}_0(G)=\mathrm{eg}(G)$ for every graph $G$; this is still an open question. They also posed the following natural ``approximate'' questions.

\begin{problem}[{\cite{SS13_block}}]
\label{problem_bounded_genus}
Is there a function $f$ such that $\mathrm{g}(G) \le f(\mathrm{g}_0(G))$ for every graph $G$? Is there a function $f$ such that $\mathrm{eg}(G) \le f(\mathrm{eg}_0(G))$ for every graph $G$?
\end{problem}

We give a positive answer to Problem~\ref{problem_bounded_genus} for several families of graphs, which we conjectured to be ``unavoidable'' as minors in graphs of large genus. Recently we have found that a similar Ramsey-type statement by Robertson and Seymour, which we formulate as Conjecture~\ref{conjecture_folklore}, is a folklore unpublished result in the graph-minors community. Together, these results would imply a positive solution to Problem~\ref{problem_bounded_genus} for all graphs.

In particular, Robertson and Seymour conjectured that every graph of a sufficiently large Euler genus contains as a minor one of the following \emph{$t$-Kuratowski graphs}: $K_{3,t}$, or $t$ copies of $K_5$ or $K_{3,3}$ sharing at most two common vertices. To obtain a similar statement for graphs of large genus, we need to add the projective $t\times t$ grid (or $t$-wall) to the list of unavoidable minors. We show that the $\mathbb{Z}_2$-genus of graphs in these families is equal to their genus. 

Our main technical tool is the intersection form over $\mathbb{Z}_2$, counting the parity of crossings between cycles on a given surface, and the fact that the rank of the intersection form is equal to the Euler genus of the surface.
 
We state the results in detail in Sections~\ref{section_ramsey_type} and~\ref{section_results} after giving necessary definitions in Section~\ref{section_definitions}.

A positive answer to Problem~\ref{problem_bounded_genus} would also have the following applications: it would give a linear upper bound on the number of edges of a graph with an independently even drawing on a fixed orientable surface, and thus imply a generalization of the crossing lemma on orientable surfaces for several notions of the crossing number, including the pair-crossing number~\cite{Ky20_issue}.

\section{Preliminaries}
\label{section_definitions}

\subsection{Graphs on surfaces}

We refer to the monograph by Mohar and Thomassen~\cite{MT01_graphs} for a detailed introduction into surfaces and graph embeddings.
By a {\em surface} we mean a connected compact $2$-dimensional topological manifold. Every surface is either {\em orientable} (has two sides) or {\em nonorientable} (has only one side). Every orientable surface $S$ is obtained from the sphere by attaching $g \ge 0$ \emph{handles}, and this number $g$ is called the {\em genus} of $S$.
Similarly, every nonorientable surface $S$ is obtained from the sphere by attaching $g \ge 1$ \emph{crosscaps}, and this number $g$ is called the {\em (nonorientable) genus} of $S$. The simplest orientable surfaces are the sphere (with genus $0$) and the torus (with genus $1$). The simplest nonorientable surfaces are the projective plane (with genus $1$) and the Klein bottle (with genus $2$). We denote the orientable surface of genus $g$ by $M_g$, and the nonorientable surface of genus $g$ by $N_g$.

Let $G=(V,E)$ be a graph with no multiple edges and no loops,
and let $S$ be a surface.  
A \emph{drawing} of $G$ on $S$ is a representation of $G$ where every vertex is represented by a unique point in $S$ and every
edge $e$ joining vertices $u$ and $v$ is represented by a simple curve in $S$ joining the two points that represent $u$ and $v$. If it leads to no confusion, we do not distinguish between
a vertex or an edge and its representation in the drawing and we use the words ``vertex'' and ``edge'' in both contexts. We require that in a drawing no edge passes through a vertex,
no two edges touch, every edge has only finitely many intersection points with other edges and no three edges cross at the same inner point. In particular, every common point of two edges is either their common endpoint or a crossing.

A drawing of $G$ on $S$ is an \emph{embedding} if no two edges cross. A \emph{face} of an embedding of $G$ on $S$ is a connected component of the topological space obtained from $S$ by removing all the edges and vertices of $G$. A \emph{$2$-cell embedding} is an embedding whose each face is homeomorphic to an open disc. The \emph{facewidth} (also called \emph{representativity}) $\mathrm{fw}(\mathcal{E})$ of an embedding $\mathcal{E}$ on a surface $S$ of positive genus is the smallest nonnegative integer $k$ such that there is a closed noncontractible curve in $S$ intersecting $\mathcal{E}$ in $k$ vertices.

The \emph{rotation} of a vertex $v$ in a drawing of $G$ on an orientable surface is the clockwise cyclic order of the edges incident to $v$. We will represent the rotation of $v$ by the cyclic order of the other endpoints of the edges incident to $v$. The {\em rotation system} of a drawing is the set of rotations of all vertices. 

The \emph{Euler characteristic} of a surface $S$ of genus $g$, denoted by $\chi(S)$, is defined as $\chi(S)=2-2g$ if $S$ is orientable, and $\chi(S)=2-g$ if $S$ is nonorientable. Equivalently, if $v$, $e$ and $f$ denote the numbers of vertices, edges and faces, respectively, of a $2$-cell embedding of a graph on $S$, then $\chi(S)=v-e+f$. The \emph{Euler genus} $\mathrm{eg}(S)$ of $S$ is defined as $2-\chi(S)$. In other words, the Euler genus of $S$ is equal to the genus of $S$ if $S$ is nonorientable, and to twice the genus of $S$ if $S$ is orientable. This implies the following inequalities for the different notions of genus of a graph $G$, defined in the introduction:

\begin{equation}
\label{eq_genusy}
\mathrm{eg}(G) \le 2\mathrm{g}(G) \ \text{ and } \ \mathrm{eg}_0(G) \le 2\mathrm{g}_0(G).
\end{equation}

An edge in a drawing is {\em even\/} if it crosses every other edge an even number of times.
A drawing of a graph is \emph{even} if all its edges are even.
A drawing of a graph is \emph{independently even} if every pair of independent edges in the drawing crosses an even number of times.
In the literature, the notion of \emph{$\mathbb{Z}_2$-embedding} is used to denote both an even drawing~\cite{CN00_thrackles} and an independently even drawing~\cite{SS13_block}.

The \emph{embedding scheme} of a drawing $\mathcal{D}$ on a surface $S$ consists of 
the rotation system and a signature $+1$ or $-1$ assigned to every edge, representing the parity of the number of crosscaps the edge is passing through. If $S$ is orientable, the embedding scheme can be given just by the 
rotation system. The following weak analogue of the Hanani--Tutte theorem was proved by Cairns and Nikolayevsky~\cite{CN00_thrackles} for orientable surfaces and then extended by
Pelsmajer, Schaefer and {\v{S}}tefankovi{\v{c}}~\cite{PSS09_surfaces} to nonorientable surfaces. Loebl and Masbaum~\cite[Theorem 5]{LoMa11_arf} obtained an alternative proof for orientable surfaces.

\begin{theorem}[The weak Hanani--Tutte theorem on surfaces{~\cite[Lemma 3]{CN00_thrackles}, \cite[Theorem 3.2]{PSS09_surfaces}}]
\label{theorem_weaksurface}
If a graph $G$ has an even drawing $\mathcal{D}$ on a surface $S$, then $G$ has an embedding on $S$ that preserves the embedding scheme of $\mathcal{D}$. 
\end{theorem}

A simple closed curve $\gamma$ in a surface $S$ is \emph{$1$-sided} if it has a small neighborhood homeomorphic to the M\"obius strip, and \emph{$2$-sided} if it has a small neighborhood homeomorphic to the cylinder. We say that $\gamma$ is \emph{separating} in $S$ if the complement $S\setminus \gamma$ has two components, and \emph{nonseparating} if $S\setminus \gamma$ is connected. Note that on an orientable surface every simple closed curve is $2$-sided, and every $1$-sided simple closed curve (on a nonorientable surface) is nonseparating.

\subsection{Special graphs}

\subsubsection{Projective grids and walls}
For a positive integer $n$ we denote the set $\{1,\dots, n\}$ by $[n]$. Let $r,s\ge 3$. The \emph{projective $r\times s$ grid} is the graph with vertex set $[r]\times[s]$ and edge set 
\[
\{\{(i,j),(i',j')\}; |i-i'|+|j-j'|=1\} \ \cup \ \{\{(i,1),(r+1-i,s)\};i\in[r]\}.
\]
In other words, the projective $r\times s$ grid is obtained from the planar $r \times (s+1)$ grid by identifying pairs of opposite vertices and edges in its leftmost and rightmost column. See Figure~\ref{obr_grid_wall}, left.
The projective $t\times t$ grid has an embedding on the projective plane with facewidth $t$. By a result of Robertson and Vitray~\cite{RoVi90_representativity},~\cite[p. 171]{MT01_graphs}, the embedding is unique if $t\ge 4$. Hence, for $t\ge 4$ the genus of the projective $t\times t$ grid is equal to $\lfloor t/2 \rfloor$ by a result of Fiedler, Huneke, Richter and Robertson~\cite{FHRR95_projective},~\cite[Theorem 5.8.1]{MT01_graphs}.

Since grids have vertices of degree $4$, it is more convenient for us to consider their subgraphs of maximum degree $3$, called walls.
For an odd $t\ge 3$, a \emph{projective $t$-wall} is obtained from the projective $t\times (2t-1)$ grid 
by removing edges $\{(i,2j),(i+1,2j)\}$ for $i$ odd and $1\le j\le t-1$, and edges 
$\{(i,2j-1),(i+1,2j-1)\}$ for $i$ even and $1\le j\le t$. Similarly, for an even $t\ge 4$, a \emph{projective $t$-wall} is obtained from the projective $t\times 2t$ grid 
by removing edges $\{(i,2j),(i+1,2j)\}$ for $i$ odd and $1\le j\le t$, and edges 
$\{(i,2j-1),(i+1,2j-1)\}$ for $i$ even and $1\le j\le t$. The projective $t$-wall has maximum degree $3$ and can be embedded on the projective plane
as a ``twisted wall'' with inner faces bounded by $6$-cycles forming the ``bricks'', and with the ``outer'' face bounded by a $(4t-2)$-cycle for $t$ odd and a $4t$-cycle for $t$ even. See Figure~\ref{obr_grid_wall}, right. This embedding has facewidth $t$ and so again, for $t\ge 4$ the projective $t$-wall has genus $\lfloor t/2 \rfloor$. It is easy to see that the projective $3$-wall has genus $1$ since it contains a subdivision of $K_{3,3}$ and embeds on the torus.

\begin{figure}
\begin{center}
\includegraphics{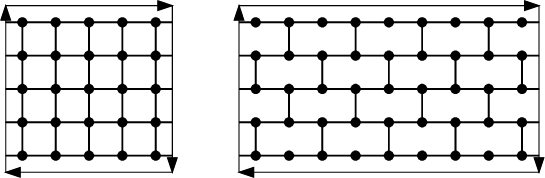}
\end{center}
\caption{Left: a projective $5\times 5$ grid. Right: a projective $5$-wall.}
\label{obr_grid_wall}
\end{figure}

\subsubsection{Kuratowski graphs}

A graph is called a \emph{$t$-Kuratowski} graph~\cite{Seym17} if it is one of the following:
\begin{enumerate}
\item[a)] $K_{3,t}$,
\item[b)] a disjoint union of $t$ copies of $K_5$,
\item[c)] a disjoint union of $t$ copies of $K_{3,3}$,
\item[d)] a graph obtained from $t$ copies of $K_5$ by identifying one vertex from each copy to a single common vertex,
\item[e)] a graph obtained from $t$ copies of $K_{3,3}$ by identifying one vertex from each copy to a single common vertex,
\item[f)] a graph obtained from $t$ copies of $K_5$ by identifying a pair of vertices from each copy to a common pair of vertices,
\item[g)] a graph obtained from $t$ copies of $K_{3,3}$ by identifying a pair of adjacent vertices from each copy to a common pair of vertices,
\item[h)] a graph obtained from $t$ copies of $K_{3,3}$ by identifying a pair of nonadjacent vertices from each copy to a common pair of vertices.
\end{enumerate}

See Figure~\ref{obr_kuratowski} for an illustration.

\begin{figure}
\begin{center}
\includegraphics{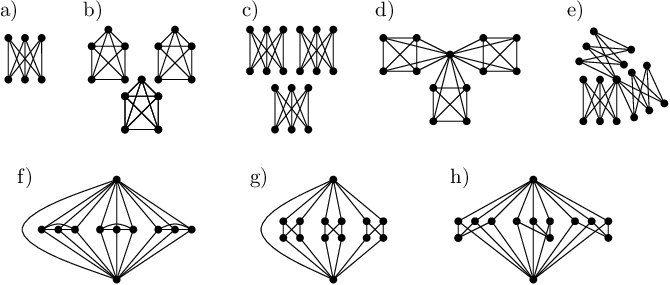}
\end{center}
\caption{The eight $3$-Kuratowski graphs.}
\label{obr_kuratowski}
\end{figure}

The genus of each of the $t$-Kuratowski graphs is known precisely. The genus of $K_{3,t}$ is $\lceil(t-2)/4\rceil$~\cite{Bo78_Kmn,Ri65_Kmn},~\cite[Theorem 4.4.7]{MT01_graphs},~\cite[Theorem 4.5.3]{GrTu01_theory}, which coincides with the lower bound from Euler's formula.
The genus of $t$ copies of $K_5$ or $K_{3,3}$ sharing at most one vertex is $t$ by the additivity of genus over blocks and connected components~\cite{BHKY62_additivity},~\cite[Theorem 4.4.2]{MT01_graphs},~\cite[Theorem 3.5.3]{GrTu01_theory}. Finally, from a general formula by Decker, Glover and Huneke~\cite{DGH85_2amalgamation} it follows that the genus of $t$ copies of $K_5$ or $K_{3,3}$ sharing a pair of adjacent or nonadjacent vertices is $\lceil t/2 \rceil$ if $t>1$: cases f) and g) follow from their proof of Corollary 0.2, case h) follows from their Corollary 2.4 after realizing that $\mu(K_{3,3})=3$ if $x,y$ are nonadjacent in $K_{3,3}$.

The Euler genus of each of the $t$-Kuratowski graphs is also known precisely.
The Euler genus of $K_{3,t}$ is $\lceil(t-2)/2\rceil$~\cite{Bo78_Kmn,Ri65b_Kmn_Euler}. The Euler genus of $t$ copies of $K_5$ or $K_{3,3}$ sharing one vertex is $t$ by the additivity of Euler genus over blocks~\cite[Corollary 2]{StBe77_blocks},~\cite[Theorem 1]{Mi87_additivity},~\cite[Theorem 4.4.3]{MT01_graphs}. The additivity of Euler genus over connected components follows almost trivially: every embedding of a disconnected graph with components $G_1$, $G_2$ on a surface can be turned into an embedding of a connected graph on the same surface by adding an edge joining $G_1$ with $G_2$. Miller~\cite[Theorem 1]{Mi87_additivity} proved that Euler genus is also additive over edge-amalgamations, which implies that the Euler genus of $t$ copies of $K_5$ or $K_{3,3}$ sharing a pair of adjacent vertices it $t$. Miller~\cite[Theorem 27]{Mi87_additivity} also proved a superadditivity of the Euler genus over $2$-amalgamations. Richter~\cite[Theorem 1]{R87_2amalgamation} proved a precise formula for the Euler genus of $2$-amalgamations with respect to a pair of nonadjacent vertices.
Since the graph obtained from $K_{3,3}$ by adding one edge has an embedding in the projective plane, Miller's~\cite{Mi87_additivity} and Richter's~\cite{R87_2amalgamation} results also imply that the Euler genus of $t$ copies of $K_{3,3}$ sharing a pair of nonadjacent vertices is $t$.

\section{Ramsey-type results}
\label{section_ramsey_type}

The following Ramsey-type statement for graphs of large Euler genus is a folklore unpublished result.

\begin{conjecture}[Robertson--Seymour{~\cite{BKMM_7connected,Seym17}}, unpublished]
\label{conjecture_folklore}
There is a function $g$ such that for every $t\ge 3$, every graph of Euler genus $g(t)$ contains a $t$-Kuratowski graph as a minor.
\end{conjecture}

For $7$-connected graphs, Conjecture~\ref{conjecture_folklore} follows from the result of B{\"o}hme, Ka\-wa\-ra\-ba\-yashi, Maharry and Mohar~\cite{BKMM_7connected}, stating that for every positive integer~$t$, every sufficiently large $7$-connected graph contains $K_{3,t}$ as a minor. B{\"o}hme et al.~\cite{BKMM_linearconnect} later generalized this to graphs of larger connectivity and $K_{a,t}$ minors for every fixed $a>3$. Fr\"ohlich and M\"uller~\cite{FrMu11_alternative} gave an alternative proof of this generalized result.

Christian, Richter and Salazar~\cite{CRS15_continuum} proved a similar statement for graph-like continua.

We obtain an analogous Ramsey-type statement for graphs of large genus as an almost direct consequence of Conjecture~\ref{conjecture_folklore}.

\begin{theorem}
\label{theorem_ramsey}
Conjecture~\ref{conjecture_folklore} implies that there is a function $h$ such that for every $t\ge 3$, every graph of genus $h(t)$ contains, as a minor, a $t$-Kuratowski graph or the projective $t$-wall.
\end{theorem}

We give a detailed proof of Theorem~\ref{theorem_ramsey} in Section~\ref{section_ramsey}.

\section{Our results}
\label{section_results}

As our main result we complete a proof that the $\mathbb{Z}_2$-genus of each $t$-Kuratowski graph and the projective $t$-wall grows to infinity with $t$; in fact, the $\mathbb{Z}_2$-genus of each of these graphs is equal to their genus. Analogously, we also show that the Euler $\mathbb{Z}_2$-genus of each $t$-Kuratowski graph is equal to its Euler genus.
Schaefer and \v{S}tefankovi\v{c}~\cite{SS13_block} proved this for those $t$-Kuratowski graphs that consist of $t$ copies of $K_5$ or $K_{3,3}$ sharing at most one vertex. For the projective $t$-wall, the result follows directly from the weak Hanani--Tutte theorem on orientable surfaces~\cite[Lemma 3]{CN00_thrackles}: indeed, all vertices of the projective $t$-wall have degree at most $3$, therefore pairs of adjacent edges crossing oddly in an independently even drawing can be redrawn in a small neighborhood of their common vertex so that they cross evenly; see Figure~\ref{obr_deg3}. Then the weak Hanani--Tutte theorem can be applied. Thus, the remaining cases are $t$-Kuratowski graphs of type a), f), g) and h).

\begin{figure}
\begin{center}
\includegraphics{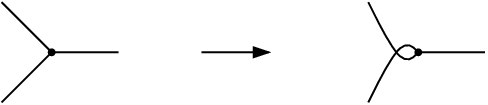}
\end{center}
\caption{Changing the parity of the number of crossings of a pair of edges incident to a vertex of degree $3$.}
\label{obr_deg3}
\end{figure}

\begin{theorem}
\label{theorem_z2}
For every $t\ge 3$, the $\mathbb{Z}_2$-genus of each $t$-Kuratowski graph of type a), f), g) and h) is equal to its genus, and also its Euler $\mathbb{Z}_2$-genus is equal to its Euler genus. 
In particular,
\begin{enumerate}
\item[{\rm a)}] $\mathrm{g}_0(K_{3,t})\ge\lceil(t-2)/4\rceil$, $\mathrm{eg}_0(K_{3,t})\ge\lceil(t-2)/2\rceil$, and
\item[{\rm b)}] if $G$ consists of $t$ copies of $K_5$ or $K_{3,3}$ sharing a pair of adjacent or nonadjacent vertices, then $\mathrm{g}_0(G) \ge\lceil t/2 \rceil$ and $\mathrm{eg}_0(G) \ge t$.
\end{enumerate}
This implies, together with the result of Schaefer and \v{S}tefankovi\v{c}~\cite{SS13_block}, that for every $t\ge 3$, the $\mathbb{Z}_2$-genus of each $t$-Kuratowski graph and the projective $t$-wall is equal to its genus, and the Euler $\mathbb{Z}_2$-genus of each $t$-Kuratowski graph is equal to its Euler genus. 
\end{theorem}

Combining Theorem~\ref{theorem_z2} with Theorem~\ref{theorem_ramsey} we get the following implication.

\begin{corollary}
Conjecture~\ref{conjecture_folklore} implies a positive answer to both parts of Problem~\ref{problem_bounded_genus}.
\end{corollary}

\section{Unavoidable minors of large genus}
\label{section_ramsey}

In this section we prove Theorem~\ref{theorem_ramsey}.

\subsection{Tools and preparations}

We will need the following classical result by Robertson and Seymour~\cite{RoSe90_VII_paths} about surface minors.
\emph{Surface minors} are defined for embeddings analogously as minors for graphs, by deleting and contracting edges on the underlying surface~\cite{MT01_graphs}.

\begin{theorem}[{\cite{RoSe90_VII_paths},~\cite[Theorem 3.5]{KaMo07_survey},~\cite[Theorem 5.2]{Mo01_survey},~\cite[Theorem 5.9.2]{MT01_graphs}}]
\label{veta_facewidth}
For every surface $S$ and every embedding $\mathcal{H}$ of a graph $H$ on $S$ there exists a constant $w(\mathcal{H},S)$ such that every embedding of a graph on $S$ with facewidth at least $w(\mathcal{H},S)$ contains $\mathcal{H}$ as a surface minor.
\end{theorem}

Let $\mathcal{W}_t$ be an embedding of the projective $t$-wall on the projective plane; see Figure~\ref{obr_grid_wall}, right. With a slight abuse of notation, for each nonorientable surface $N_i$ with $i\ge 2$, we choose an embedding of the projective $t$-wall on $N_i$ and denote it again by $\mathcal{W}_t$. 
Without loss of generality, we will assume that $w(\mathcal{W}_t,N_i)$ is nondecreasing in $i$; otherwise we inductively redefine $w(\mathcal{W}_t,N_i)$ as $\max\{w(\mathcal{W}_t,N_j); j\le i\}$.
For all integers $k',i,k$ satisfying $0\le 2k'<i\le k$, let 
\[
w(k',i,k,t)=i(i-2k')\cdot(w(\mathcal{W}_t,N_i) + 2k).
\]
This function will be used as a ``potential function'' in the proof of Proposition~\ref{prop_wall}.

We will also use the following simple statement about the ``continuity'' of facewidth under the operation of removing all vertices of a face.

\begin{proposition}[{\cite[Propositions 5.5.7 and 5.5.8]{MT01_graphs}}]
\label{prop_face}
Let $\mathcal{E}$ be an embedding of a graph on a surface $S$ with $\mathrm{fw}(\mathcal{E})\ge 3$. Let $f$ be a face of $\mathcal{E}$ and let $\mathcal{E}'$ be the embedding obtained from $\mathcal{E}$ by removing all vertices incident to $f$. Then $\mathrm{fw}(\mathcal{E}')\ge \mathrm{fw}(\mathcal{E}) - 2$.
\end{proposition}

\subsection{Proof of Theorem~\ref{theorem_ramsey}}
\label{subsection_zed}

Let $t\ge 3$ and let $g$ be a sufficiently large integer, larger than $g(t)/2$ where $g(t)$ is the number from Conjecture~\ref{conjecture_folklore}. Let $G$ be a graph of genus $g$. If the Euler genus of $G$ is larger than $g(t)$, then $G$ has a $t$-Kuratowski minor by Conjecture~\ref{conjecture_folklore}. For the rest of the proof we thus assume that the Euler genus of $G$ is at most $k=g(t)$, and our goal is to find the projective $t$-wall as a minor in $G$. Since $2g>k$, this implies that $G$ has an embedding $\mathcal{E}$ on $N_k$.

The operation of \emph{gluing a pair of vertices} $u,v$ in a graph $G$ creates a graph with vertex set $V(G)\setminus\{u,v\} \cup \{w\}$, where $w\notin V(G)$, and edge set $E(G[V(G)\setminus\{u,v\}]) \cup \{\{w,x\}; \{u,x\}\in E(G)\}\} \cup \{\{w,x\}; \{v,x\}\in E(G)\}$. We emphasize that this gluing operation creates no loops or multiple edges. An inverse operation is called \emph{splitting a vertex}; in general, this is not unique for a given graph and a vertex.

We show the following proposition by induction on $i$.

\begin{proposition}
\label{prop_wall}
Let $i,k,t$ be positive integers with $t\ge 3$ and $i\le k$. Let $G$ be a graph that has an embedding $\mathcal{E}$ on $N_i$, let $F$ be a set of at most $k-i$ faces in $\mathcal{E}$, and let $Z$ be the set of all vertices of $\mathcal{E}$ incident to at least one face in $F$. Then at least one of the following holds:
\begin{enumerate}
\item[{\rm 1)}] $G-Z$ has a projective $t$-wall as a minor, or
\item[{\rm 2)}] there is an integer $k'$ satisfying $0\le 2k'<i$ such that $G$ can be obtained from a graph $H$ of genus at most $k'$ 
by at most $w(k',i,k,t)$ consecutive operations of gluing a pair of vertices (shortly \emph{gluings}).
\end{enumerate}
\end{proposition}

\begin{proof}
The main idea of the proof is to cut the surface recursively along ``short'' noncontractible curves until we obtain an embedding of large facewidth on a nonorientable surface, or until all the pieces are orientable.

We distinguish two cases according to the facewidth of $\mathcal{E}$.

1) $\mathrm{fw}(\mathcal{E}) \ge w(\mathcal{W}_t,N_i)+2(k-i)$. By Proposition~\ref{prop_face}, the induced embedding $\mathcal{E}'$ of $G-Z$ in $\mathcal{E}$ has facewidth at least $w(\mathcal{W}_t,N_i)$. Thus, $\mathcal{W}_t$ is a surface minor of $\mathcal{E}'$ and so the projective $t$-wall is a minor of $G-Z$.

2) $\mathrm{fw}(\mathcal{E}) < w(\mathcal{W}_t,N_i)+2(k-i)$. In this case there is a noncontractible closed curve $\gamma$ on $S$ intersecting $\mathcal{E}$ in less than $w(\mathcal{W}_t,N_i)+2(k-i)$ points, all of which can be assumed to be vertices. Let $W$ be the set of the vertices in $\mathcal{E}\cap \gamma$. We have three cases according to the type of $\gamma$: a) $\gamma$ is 1-sided, b) $\gamma$ is 2-sided but nonseparating in $N_i$, c) $\gamma$ is 2-sided and separates $N_i$ into two components.

In each case, we cut $N_i$ along $\gamma$, obtaining a surface or a pair of surfaces with boundary, and fill the boundary cycles with discs. The resulting surfaces may be orientable or nonorientable. In case a) we obtain a surface $S$ of Euler genus $i-1$. In case b) we obtain a surface $S$ of Euler genus $i-2$. In case c) we obtain a pair of surfaces $S_1$ and $S_2$ with Euler genera $i_1$ and $i_2$, respectively, such that $i_1+i_2=i$ and $1\le i_1,i_2\le i-1$.

While cutting the surface $N_i$ along $\gamma$, we also obtain an embedding $\mathcal{E}'$ of a graph $G'$ on $S$ or a pair of embeddings $\mathcal{E}_1'$ and $\mathcal{E}_2'$ of $G'_1$ and $G'_2$ on $S_1$ and $S_2$, respectively, obtained from $\mathcal{E}$ by splitting each vertex in $W$ into two copies, each copy keeping adjacent edges only from one side of $\gamma$. We now consider each of the three cases separately.

In case a), the embedding $\mathcal{E}'$ has one new face $f$, containing the disc that was used to fill a boundary cycle while creating $S$. On the other hand, the faces of $\mathcal{E}$ whose interior intersects $\gamma$ are no longer faces of $\mathcal{E}'$, as they were cut and merged into $f$. Let $F'$ be the union of $\{f\}$ and the subset of faces in $F$ that are still faces of $\mathcal{E}'$. Clearly, we have $|F'|\le |F|+1 \le k-i+1 = k-\mathrm{eg}(S)$.

If $S$ is orientable, then $G'$ is a graph of genus at most $\mathrm{eg}(S)/2=(i-1)/2$ and $G$ can be obtained from $G'$ by less than $w(\mathcal{W}_t,N_i)+2(k-i)\le w((i-1)/2,i,k,t)$ gluings.

If $S$ is nonorientable, we apply induction to the embedding $\mathcal{E}'$ and the set of faces $F'$. Let $Z'$ be the set of vertices of $\mathcal{E}'$ incident with at least one face in $F'$. Observe that $Z'$ contains all vertices in $Z$ and also all  new vertices created by splitting the vertices in $W$. Hence, $G'-Z'$ is a subgraph of $G-Z$. Therefore, if case 1) of the proposition occurs, we obtain a projective $t$-wall as a minor in both $G'-Z'$ and $G-Z$. In case 2) we obtain $G'$ from a graph $H$ of genus $k'<(i-1)/2$ by at most $w(k',i-1,k,t)$ gluings. Since $G$ is obtained from $G'$ by less than $w(\mathcal{W}_t,N_i)+2(k-i)$ gluings, we can obtain $G$ from $H$ by less than $w(k',i,k,t)$ gluings.

In case b), $\mathcal{E}'$ has two new faces $f_1$ and $f_2$. Let $F'$ be the union of $\{f_1,f_2\}$ and the subset of faces in $F$ that are still faces of $\mathcal{E}'$. Clearly, we have $|F'|\le |F|+2 \le k-i+2 = k-\mathrm{eg}(S)$.

If $S$ is orientable, then $G'$ is a graph of genus at most $\mathrm{eg}(S)/2=(i-2)/2$ and $G$ can be obtained from $G'$ by less than $w(\mathcal{W}_t,N_i)+2(k-i)\le w((i-2)/2,i,k,t)$ gluings.

If $S$ is nonorientable, we apply induction to $\mathcal{E}'$ and $F'$ and proceed analogously as in case a). If case 2) of the proposition occurs, we obtain $G'$ from a graph $H$ of genus $k'<(i-2)/2$ by at most $w(k',i-2,k,t)$ gluings, and thus we can again obtain $G$ from $H$ by less than $w(k',i,k,t)$ gluings.

In case c), $\mathcal{E}'_1$ has a new face $f_1$ and $\mathcal{E}'_2$ has a new face $f_2$. For $l\in\{1,2\}$ we define $F'_l$ as the union of $\{f_l\}$ and the the subset of faces in $F$ that are still faces of $\mathcal{E}'_l$. Again, for each $l\in\{1,2\}$ we have $|F'_l|\le |F|+1 \le k-i+1 \le k-\mathrm{eg}(S_l)$.

Notice that at least one of the surfaces $S_1, S_2$ is nonorientable, since $N_i$ is their connected sum. 
Let $l\in\{1,2\}$. If $S_l$ is orientable, then $G'_l$ is a graph of genus at most $\mathrm{eg}(S_l)/2=i_l/2$. If $S_l$ is nonorientable, we apply induction to $\mathcal{E}'_l$ and $F'_l$. Let $Z'_l$ be the set of vertices of $\mathcal{E}'_l$ incident with at least one face in $F'_l$. Observe that $Z'_l$ contains all vertices in $Z\cap V(G'_l)$ and all new vertices in $G'_l$ created by splitting the vertices in $W$. Hence, $G'_l-Z'_l$ is a subgraph of $G-Z$. Therefore, if case 1) of the proposition occurs, we obtain a projective $t$-wall as a minor in both $G'_l-Z'_l$ and $G-Z$. In case 2) we obtain $G'_l$ from a graph $H_l$ of genus $k'_l<i_l/2$ by at most $w(k'_l,i_l,k,t)$ gluings. 

If we have not obtained the projective $t$-wall as a minor in $G-Z$, then for each $l\in\{1,2\}$, the graph $G'_l$ is obtained from a graph $H_l$ of genus $k'_l\le i_l/2$ by at most $w(k'_l,i_l,k,t)$ gluings (where $w(i_l/2,i_l,k,t)=0$), and $k'_1+k'_2 \le (i-1)/2$ since at least one of $S_1, S_2$ is nonorientable. Let $H$ be the disjoint union of $H_1$ and $H_2$. Then $H$ is a graph of genus at most $k'=k'_1+k'_2<i/2$, and $G$ can be obtained from $H$ by less than $w(k'_1,i_1,k,t)+w(k'_2,i_2,k,t)+w(\mathcal{W}_t,N_i)+2(k-i)$ gluings. By the monotonicity of $w(\mathcal{W}_t,N_i)$, we have 
\begin{align*}
&w(k'_1,i_1,k,t)+w(k'_2,i_2,k,t)+w(\mathcal{W}_t,N_i)+2(k-i) \\
&\le (i_1(i_1-2k'_1)+i_2(i_2-2k'_2)+1)\cdot (w(\mathcal{W}_t,N_i) + 2k)\\ 
&\le (i(i-2k') + 1 - i_1(i_2-2k'_2) - i_2(i_1-2k'_1)) \cdot (w(\mathcal{W}_t,N_i) + 2k)\\
&\le w(k',i,k,t).
\end{align*}
This finishes the proof of the proposition.
\end{proof}

We apply Proposition~\ref{prop_wall} with $i=k$ and $F=\emptyset=Z$. If case 1) occurs, then $G$ has the projective $t$-wall as a minor. If case 2) occurs, then there is an integer $k'$ satisfying $0\le 2k'<k$ such that $G$ can be obtained from a graph $H$ of genus at most $k'$ by at most $w(k',k,k,t)$ gluings. Since every gluing increases the genus of a graph by at most $1$, we conclude that the genus of $G$ is at most $k'+w(k',k,k,t) \le k^2\cdot(w(\mathcal{W}_t,N_k) + 2k)$. This will be a contradiction if $g>k^2\cdot(w(\mathcal{W}_t,N_k) + 2k)$. Therefore, in Theorem~\ref{theorem_ramsey} it is sufficient to take $h(t)=g^2(t)\cdot(w(\mathcal{W}_t,N_{g(t)}) + 2g(t))$ where $g(t)$ is the number from Conjecture~\ref{conjecture_folklore}.

\section{\texorpdfstring{Lower bounds on the $\mathbb{Z}_2$-genus and Euler $\mathbb{Z}_2$-genus}{Lower bounds on the Z2-genus and Euler Z2-genus}}

In this section we prove Theorem~\ref{theorem_z2}. By~\eqref{eq_genusy}, the lower bounds on the Euler $\mathbb{Z}_2$-genus 
of the $t$-Kuratowski graphs in Theorem~\ref{theorem_z2} imply the lower bounds on their $\mathbb{Z}_2$-genus; thus it will be sufficient to prove the lower bounds on their Euler $\mathbb{Z}_2$-genus.

The fact that the (Euler) $\mathbb{Z}_2$-genus of $K_{3,t}$ or the other $t$-Kuratowski graphs is unbounded when $t$ goes to infinity is not obvious at first sight. The traditional lower bound on the (Euler) genus of $K_{3,t}$ relies on Euler's formula and the notion of a face. However, there is no analogue of a ``face'' in an independently even drawing, and the rotations of vertices no longer ``matter''. We thus need different tools to compute the (Euler) $\mathbb{Z}_2$-genus.

\subsection{\texorpdfstring{$\mathbb{Z}_2$-homology of curves}{Z2-homology of curves}}

We refer to Hatcher's textbook~\cite{Hat02_AT} for an excellent general introduction to homology theory. Unfortunately, except for the very short summary by Colin de Verdi\`ere~\cite[p. 14--15]{Code17_computational}, we were unable to find a compact treatment of the  homology theory for curves on surfaces in the literature, thus we sketch here the main aspects that are most important for us.

We will use the $\mathbb{Z}_2$-homology of closed curves on surfaces. That is, for a given surface $S$, we are interested in its first homology group with coefficients in $\mathbb{Z}_2$, denoted by $H_1(S;\mathbb{Z}_2)$. It is well-known that for each $g\ge 0$, the first homology group $H_1(M_g;\mathbb{Z}_2)$ of $M_g$ is isomorphic to $\mathbb{Z}_2^{2g}$~\cite[Example 2A.2. and Corollary 3A.6.(b)]{Hat02_AT}. This fact was crucial in establishing the weak Hanani--Tutte theorem on $M_g$~\cite[Lemma 3]{CN00_thrackles}. Similarly, for each $g\ge 1$, the first homology group $H_1(N_g;\mathbb{Z}_2)$ of $N_g$ is isomorphic to $\mathbb{Z}_2^{g}$~\cite[Example 2.37 and Corollary 3A.6.(b)]{Hat02_AT}.

To every closed curve $\gamma$ in a surface $S$ one can assign its homology class $[\gamma] \in H_1(S;\mathbb{Z}_2)$, and this assignment is invariant under continuous deformation (homotopy). In particular, the homology class of each contractible curve is $0$. More generally, the homology class of each separating curve in $S$ is $0$ as well. Moreover, if $\gamma$ is obtained by a composition of $\gamma_1$ and $\gamma_2$, the homology classes satisfy $[\gamma]=[\gamma_1]+[\gamma_2]$. The assignment of homology classes to closed curves is naturally extended to formal integer combinations of the closed curves, called \emph{cycles}, and so $[\gamma]$ can be considered as a set of cycles. Since we are interested in homology with coefficients in $\mathbb{Z}_2$, it is sufficient to consider cycles with coefficients in $\mathbb{Z}_2$, which may also be regarded as finite sets of closed curves. 

If $\gamma_1$ and $\gamma_2$ are cycles in $S$ that cross in finitely many points and have no other points in common, we denote by $\mathrm{cr}(\gamma_1,\gamma_2)$ the number of their common crossings. We use the following well-known fact, which formalizes the intuition that by a continuous deformation of a closed curve, we can change its number of crossings with another closed curve only by an even number.

\begin{fact}
\label{fact_intersection}
Let $\gamma'_1 \in [\gamma_1]$ and $\gamma'_2 \in [\gamma_2]$ be a pair of cycles in a surface $S$ such that the intersection numbers $\mathrm{cr}(\gamma_1,\gamma_2)$ and $\mathrm{cr}(\gamma'_1,\gamma'_2)$ are defined and finite. Then 
\[
\mathrm{cr}(\gamma'_1,\gamma'_2) \equiv \mathrm{cr}(\gamma_1,\gamma_2) \ (\mathrm{mod }\ 2).
\]
\end{fact}

Fact~\ref{fact_intersection} allows us to define a bilinear form
\[
\Omega_{S}: H_1(S;\mathbb{Z}_2) \times H_1(S;\mathbb{Z}_2) \rightarrow \mathbb{Z}_2
\]
such that 
\[
\Omega_{S} ([\gamma_1],[\gamma_2]) = \mathrm{cr}(\gamma_1,\gamma_2) \ \mathrm{mod} \ 2
\] 
whenever $\mathrm{cr}(\gamma_1,\gamma_2)$ is defined and is finite. The form $\Omega_{S}$ is called the \emph{intersection form} on $S$. See e.g. Hausmann~\cite[Section 5.3.3 and Corollary 5.4.13]{Hau14_mod2}. 
Clearly, $\Omega_{S}$ is symmetric, and for every $2$-sided simple closed curve $\gamma$ we have $\Omega_{S}([\gamma],[\gamma])=0$. This implies that for every cycle $\gamma$ in an orientable surface $M_g$ we have $\Omega_{M_g}([\gamma],[\gamma])=0$, since all simple closed curves in $M_g$ are $2$-sided, and every closed curve with finitely many self-intersections can be decomposed into finitely many simple closed curves. On the other hand, if $\gamma$ is a $1$-sided simple closed curve in $N_g$, then $\Omega_{N_g}([\gamma],[\gamma])=1$.
The following fact can be verified by choosing a ``standard'' basis of $H_1(S;\mathbb{Z}_2)$.

\begin{fact}
\label{fact_full_rank}
For every surface $S$, the intersection form $\Omega_{S}$ is nondegenerate. In particular, the rank of $\Omega_{M_g}$ is $2g$ and the rank of $\Omega_{N_g}$ is $g$. In each case, the rank of $\Omega_{S}$ is equal to the Euler genus of $S$.
\end{fact}

In our proofs we need only the trivial inequality that the rank of $\Omega_{S}$ is at most the rank of $H_1(S;\mathbb{Z}_2)$, which equals the Euler genus of $S$.

We have the following simple observation about intersections of disjoint cycles in independently even drawings.

\begin{observation}[{\cite[Lemma 1]{SS13_block}}]
\label{obs_disjoint_cycles}
Let $\mathcal{D}$ be an independently even drawing of a graph $G$ on a surface $S$. Let $C_1$ and $C_2$ be vertex-disjoint cycles in $G$, and let $\gamma_1$ and $\gamma_2$ be the closed curves representing $C_1$ and $C_2$, respectively, in $\mathcal{D}$. Then $\mathrm{cr}(\gamma_1,\gamma_2) \equiv 0 \ (\mathrm{mod }\ 2)$, which implies that $\Omega_{S}([\gamma_1],[\gamma_2])=0$. 
\qed
\end{observation}

\subsection{\texorpdfstring{Combinatorial representation of the $\mathbb{Z}_2$-homology of drawings}{Combinatorial representation of the Z2-homology of drawings}}

Schaefer and \v{S}tefankovi\v{c}~\cite{SS13_block} used the following combinatorial representation of drawings of graphs on $M_g$ and $N_g$. First, every drawing of a graph on $M_g$ can be considered as a drawing on the nonorientable surface $N_{2g+1}$, since $M_g$ minus a point is homeomorphic to an open subset of $N_{2g+1}$. The surface $N_{h}$ minus a point can be represented combinatorially as the plane with $h$ \emph{crosscaps}. A crosscap at a point $x$ is a combinatorial representation of a M\"obius strip whose boundary is identified with the boundary of a small circular hole centered in $x$. Informally, the main ``objective'' of a crosscap is to allow a set of curves intersect transversally at $x$ without counting it as a crossing.

Every closed curve $\gamma$ drawn in the plane with $h$ crosscaps is assigned a vector $y_{\gamma}\in\{0,1\}^{h}$ such that $(y_{\gamma})_i=1$ if and only if $\gamma$ passes an odd number of times through the $i$th crosscap. When $\gamma$ represents a $2$-sided curve in a surface $S$, then $y_{\gamma}$ has an even number of coordinates equal to $1$. The vectors $y_{\gamma}$ represent the elements of the homology group $H_1(S;\mathbb{Z}_2)$, and the value of the intersection form 
$\Omega_{S} ([\gamma],[\gamma'])$ is equal to the scalar product $y_{\gamma}^{\top}y_{\gamma'}$ over $\mathbb{Z}_2$. Analogously, we assign a vector $y^{\mathcal{D}}_e$ (or simply $y_e$) to every curve representing an edge $e$ in a drawing $\mathcal{D}$ of a graph in this model.

\begin{figure}
\begin{center}
\includegraphics{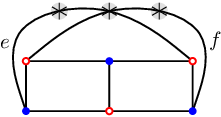}
\end{center}
\caption{An embedding of $K_{3,3}$ on the torus represented as a drawing $\mathcal{D}$ in the plane with three crosscaps. The nonzero vectors assigned to the edges are $y^{\mathcal{D}}_e=(1,1,0)$ and $y^{\mathcal{D}}_f=(0,1,1)$.}
\label{obr_k33_on_N3}
\end{figure}

We use the following two lemmata by Schaefer and \v{S}tefankovi\v{c}~\cite{SS13_block}. Here we consider \emph{crosscap drawings}, in which we allow self-intersections of edges and crossing of more than two edges in the points representing the crosscaps.

\begin{lemma}[{\cite[Lemma 5]{SS13_block}}]
\label{lemma_forest}
Let $G$ be a graph that has an independently even drawing $\mathcal{D}$ on a surface $S$ and let $F$ be a forest in $G$. Let $h=2g+1$ if $S=M_g$ and $h=g$ if $S=N_g$. Then $G$ has a crosscap drawing $\mathcal{E}$ in the plane with $h$ crosscaps, such that
\begin{enumerate}
\item[{\rm 1)}] every pair of independent edges has an even number of common crossings except those at the crosscaps, and
\item[{\rm 2)}] every edge $f$ of $F$ passes through each crosscap an even number of times; that is, $y^{\mathcal{E}}_f=0$.
\end{enumerate}
Moreover, the drawing in $S$ corresponding to $\mathcal{E}$ can be obtained from $\mathcal{D}$ by a sequence of continuous deformations of edges and neighborhoods of vertices, so the homology classes of all cycles are preserved between the two drawings.
\end{lemma}

\begin{lemma}[{\cite[Lemma 4]{SS13_block}}]
\label{lemma_smalleulergenus}
Let $G$ be a graph that has a crosscap drawing $\mathcal{D}$ in the plane with finitely many crosscaps with every pair of independent edges having an even number of common crossings except those at the crosscaps. Let $d$ be the dimension of the vector space generated by the set $\{y^{\mathcal{D}}_e$; $e\in E(G)\}$. Then $G$ has an independently even drawing on a surface of Euler genus $d$.
\end{lemma}

Lemma~\ref{lemma_forest} and Lemma~\ref{lemma_smalleulergenus} imply the following corollary generalizing the strong Hanani--Tutte theorem.

\begin{corollary}
\label{cor_zero}
Let $G$ be a connected graph with an independently even drawing on a surface $S$ such that each cycle in the drawing is homologically zero (that is, the homology class of the corresponding closed curve is $0$). Then $G$ is planar.
\end{corollary}

\begin{proof}
Let $F$ be a spanning tree of $G$ and let $\mathcal{E}$ be a crosscap drawing obtained from Lemma~\ref{lemma_forest}. The cycle space of $G$ is generated by the fundamental cycles with respect to $F$. Every edge $e\in E(G)\setminus E(F)$ determines a unique fundamental cycle $C_e\subseteq F\cup \{e\}$. Since $y^{\mathcal{E}}_f=0$ for every edge $f$ of $F$, the homology class of $C_e$ in $\mathcal{E}$ is represented by $y^{\mathcal{E}}_e$. Therefore, under the assumption that the homology classes of all cycles are zero, we have $y^{\mathcal{E}}_e=0$ for every edge $e$ of $G$. Lemma~\ref{lemma_smalleulergenus} then implies that $G$ has an independently even drawing in the plane. Finally, $G$ is planar by the strong Hanani--Tutte theorem (Theorem~\ref{theorem_strong}).
\end{proof}

Corollary~\ref{cor_zero} can be further strengthened using Lemma~\ref{lemma_forest} as follows.

\begin{lemma}
\label{lemma_iocr_cycles}
Let $G$ be a connected graph with an independently even drawing $\mathcal{D}$ on a surface $S$. Let $F$ be a spanning tree of $G$. If $G$ is nonplanar, then there are independent edges $e,f \in E(G)\setminus E(F)$ such that the closed curves $\gamma_e$ and $\gamma_f$ representing the fundamental cycles of $e$ and $f$, respectively, satisfy $\Omega_{S}([\gamma_e],[\gamma_f])=1$.
\end{lemma}

\begin{proof}
Let $\mathcal{E}$ be a crosscap drawing of $G$ from Lemma~\ref{lemma_forest}. By the strong Hanani--Tutte theorem, there are two independent edges $e$ and $f$ in $G$ that cross an odd number of times in $\mathcal{E}$. Moreover, conditions 1) and 2) of Lemma~\ref{lemma_forest} imply that none of the edges $e$ and $f$ is in $F$ and so $e$ and $f$ cross an odd number of times in the crosscaps. This means that $y_e^{\top}y_f=1$, which is equivalent to $\Omega_{S}([\gamma_e],[\gamma_f])=1$.
\end{proof}

\subsection{Proof of Theorem~\ref{theorem_z2}a)}

We will show three lower bounds on $\mathrm{g}_0(K_{3,t})$ and $\mathrm{eg}_0(K_{3,t})$, in the order of increasing strength and complexity of their proof.
The reader interested in the strongest lower bounds can skip Proposition~\ref{prop_ramsey} and Proposition~\ref{prop_K3t_log}.

We will adopt the following notation for the vertices of $K_{3,t}$. The vertices of degree $t$ forming one part of the bipartition are denoted by $a,b,c$, and the remaining vertices by $u_0, u_1, \dots, u_{t-1}$. Let $U=\{u_0, u_1, \dots, u_{t-1}\}$. For each $i\in[t-1]$, let $C_i$ be the cycle $au_ibu_0$ and $C'_i$ the cycle $au_icu_0$.

The first lower bound follows from Ramsey's theorem and the weak Hanani--Tutte theorem on surfaces.

\begin{proposition}
\label{prop_ramsey}
We have $2\mathrm{g}_0(K_{3,t}) \ge \mathrm{eg}_0(K_{3,t}) \ge \Omega(\log\log\log t)$.
\end{proposition}

\begin{proof}
Let $t\ge 3, g\ge 0$ and let $\mathcal{D}$ be an independently even drawing of $K_{3,t}$ on a surface $S$ of Euler genus $g$. Construct an auxiliary graph $G_U$ with vertex set $U$ such that $u_iu_j$ is an edge of $G_U$ if and only if the edges $au_i$ and $au_j$ of $K_{3,t}$ cross an odd number times in $\mathcal{D}$. By Ramsey's theorem (see e.g. Diestel~\cite[Section 9.1]{Di16_fifth} or Matou\v{s}ek--Ne\v{s}et\v{r}il~\cite[Theorem 11.2.1]{MN09_invitation}) applied to $G_U$, there is a subset $U_a\subseteq U$ of size $\Omega(\log t)$ such that all the edges between $a$ and $U_a$ cross each other an odd number of times, or all the edges between $a$ and $U_a$ cross each other an even number of times.
Repeating the same argument with vertices $b$ and $c$, we find a subset $U_b\subseteq U_a$ of size $\Omega(\log \log t)$, and a subset $U_c\subseteq U_b$ of size $\Omega(\log \log \log t)$ such that the number of crossings of each pair of edges between $b$ and $U_b$ has the same parity, and the number of crossings of each pair of edges between $c$ and $U_c$ has the same parity. If the parity is odd for some of the vertices $a,b,c$, we modify the drawing locally around this vertex by introducing one more crossing for each pair of incident edges; see Figure~\ref{obr_vertex_flip} or~\cite[Fig. 4]{CN00_thrackles}. Finally, as each vertex $u$ of $U_c$ has degree $3$, we modify the drawing locally around $u$ so that again, every pair of the three edges incident to $u$ crosses an even number of times; see Figure~\ref{obr_deg3}. After these modifications we obtain an even drawing of the complete bipartite graph induced by $\{a,b,c\} \cup U_c$. By the weak Hanani--Tutte theorem for surfaces (Theorem~\ref{theorem_weaksurface}), the graph $K_{3,|U_c|}$ has an embedding on $S$ and so $g\ge \lfloor(|U_c|-2)/2\rfloor$. It follows that $g \ge \Omega(\log\log\log t)$.
\end{proof}

\begin{figure}
\begin{center}
\includegraphics{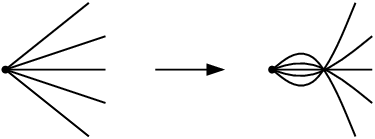}
\end{center}
\caption{Flipping the neighborhood of a vertex changes the parity of the number of crossings between any pair of incident edges.}
\label{obr_vertex_flip}
\end{figure}

The second lower bound is based on the pigeonhole principle and Corollary~\ref{cor_zero} from the previous subsection.

\begin{proposition}
\label{prop_K3t_log}
We have $2\mathrm{g}_0(K_{3,t}) \ge \mathrm{eg}_0(K_{3,t}) \ge \Omega(\log t)$.
\end{proposition}

\begin{proof}
Let $\mathcal{D}$ be an independently even drawing of $K_{3,t}$ on a surface $S$ of Euler genus $g$. By the pigeonhole principle, there is a subset $I_b\subseteq [t-1]$ of size at least $(t-1)/2^{g}$ such that all the cycles $C_i$ with $i\in I_b$ have the same homology class in $\mathcal{D}$. Analogously, there is a subset $I_c\subseteq I_b$ of size at least $|I_b|/2^{g}$ such that all the cycles $C'_i$ with $i\in I_c$ have the same homology class in $\mathcal{D}$. Suppose that $t\ge 2\cdot 4^g +2$. Then $|I_b|\ge 2\cdot 2^g +1$ and $|I_c|\ge 3$. Let $i,j,k \in I_c$ be three distinct integers. We now consider the subgraph $H$ of $K_{3,t}$ induced by the vertices $a,b,c,u_i,u_j,u_k$, isomorphic to $K_{3,3}$, and show that all its cycles are homologically zero. Indeed, the cycle space of $H$ is generated by the four cycles $au_ibu_j$, $au_ibu_k$, $au_icu_j$ and $au_icu_k$, and each of them is the sum (mod $2$) of two cycles of the same homology class: $au_ibu_j = C_i + C_j$, $au_ibu_k = C_i + C_k$, $au_icu_j = C'_i + C'_j$ and $au_icu_k = C'_i + C'_k$. Corollary~\ref{cor_zero} now implies that $H$ is planar, but this is a contradiction. Therefore $t\le 2\cdot 4^g + 1$.
\end{proof}

To prove the lower bound in Theorem~\ref{theorem_z2}a), we use the same general idea as in the previous proof. However, we will need the following more precise lemma about drawings of $K_{3,3}$, strengthening Corollary~\ref{cor_zero} and Lemma~\ref{lemma_iocr_cycles}. We also replace the pigeonhole principle with an argument from linear algebra.

\begin{lemma}
\label{lemma_K33}
Let $\mathcal{D}$ be an independently even drawing of $K_{3,3}$ on a surface $S$. For $i\in\{1,2\}$, let $\gamma_i$ and $\gamma'_i$ be the closed curves representing the cycles $C_i$ and $C'_i$, respectively, in $\mathcal{D}$. The intersection numbers of their homology classes satisfy
\[\Omega_{S}([\gamma_1],[\gamma'_2])+\Omega_{S}([\gamma'_1],[\gamma_2])=1.
\]
\end{lemma}

Lemma~\ref{lemma_K33} is a consequence of Corollary~\ref{cor_kuratowski_forest}. Here we include a direct proof using a different method.

\begin{proof}
Since the maximum degree of $K_{3,3}$ is $3$, we may assume that the drawing $\mathcal{D}$ is even: if some adjacent edges cross oddly, we may modify the drawing locally around their common vertex so that they cross evenly (see Figure~\ref{obr_deg3}), without changing the values of the intersection form.

Cairns and Nikolayevsky~\cite[Lemma 1]{CN00_thrackles} formulated a special case of an identity expressing the intersection form $\Omega_{S}$ as the sum of a ``combinatorial'' crossing number of cycles and the number of crossings of their edges. We use an analogous identity for the drawing $\mathcal{D}$, and also include its derivation to make the proof self-contained. 

\begin{figure}
\begin{center}
\includegraphics{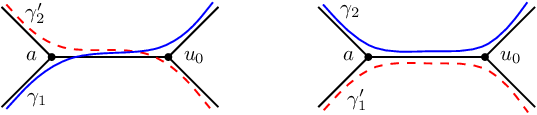}
\end{center}
\caption{The curves $\gamma_1,\gamma'_2,\gamma'_1,\gamma_2$ after deformation in the neighborhood of their common edge $au_0$.}
\label{obr_cis_trans}
\end{figure}

The cycles $C_1$ and $C'_2$ share only the vertices $a$ and $u_0$ and the edge $au_0$, and the same is true for the cycles $C'_1$ and $C_2$. Let $O$ be a small neighborhood of the curve representing the edge $au_0$ in $\mathcal{D}$. Deform the curves $\gamma_1,\gamma_2, \gamma'_1,\gamma'_2$ within $O$ so that they cross each other at most once in $O$; see Figure~\ref{obr_cis_trans}. 
Assume without loss of generality that the rotation of $a$ in $\mathcal{D}$ is $(u_0,u_1,u_2)$, the rotation of $u_0$ in $\mathcal{D}$ is $(a,b,c)$, and that the signature of the edge $au_0$ is positive if $S$ is not orientable. Then the curves obtained by deforming $\gamma_1$ and $\gamma'_2$ cross exactly once in $O$, and the curves obtained by deforming $\gamma'_1$ and $\gamma_2$ do not intersect in $O$. All the other crossings between these closed curves coincide with the crossings between edges in $\mathcal{D}$. Since $\mathcal{D}$ is an even drawing, the value of the intersection form is determined by the parity of the number of crossings inside $O$. In particular, we have $\Omega_{S}([\gamma_1],[\gamma'_2])=1$ and $\Omega_{S}([\gamma'_1],[\gamma_2])=0$.
\end{proof}

\begin{proposition}
We have $\mathrm{g}_0(K_{3,t}) \ge \lceil(t-2)/4\rceil$ and $\mathrm{eg}_0(K_{3,t}) \ge \lceil(t-2)/2\rceil$.
\end{proposition}

\begin{proof}
Let $\mathcal{D}$ be an independently even drawing of $K_{3,t}$ on a surface $S$ of Euler genus $g$. For every $i\in[t-1]$, let $\gamma_i$ and $\gamma'_i$ be the closed curves representing the cycles $C_i$ and $C'_i$, respectively, in $\mathcal{D}$.
For every $i,j\in[t-1]$, $i<j$, we apply Lemma~\ref{lemma_K33} to the drawing of $K_{3,3}$ induced by the vertices $a,b,c,u_0,u_i,u_j$ in $\mathcal{D}$. Let $A$ be the $(t-1)\times(t-1)$ matrix with entries 
\[
A_{i,j}=\Omega_{S}([\gamma_i],[\gamma'_j]).
\]
Lemma~\ref{lemma_K33} implies that $A_{i,j}+A_{j,i}=1$ whenever $i\neq j$; in other words, $A$ is the sum of a \emph{tournament matrix} and a diagonal matrix. This implies that $A+A^{\top}$, with the addition mod $2$, is a matrix with zeros on the diagonal and $1$-entries elsewhere. De Caen~\cite{Ca91_tournament} shows, by a simple argument, that the rank of $A$ over $\mathbb{Z}_2$ is at least $(t-2)/2$. 
Hence, the rank of $\Omega_{S}$ is at least $(t-2)/2$, which implies $g\ge (t-2)/2$ by Fact~\ref{fact_full_rank}. 
\end{proof}

\subsection{Proof of Theorem~\ref{theorem_z2}b)}

Before proving Theorem~\ref{theorem_z2}b) we first show an asymptotic $\Omega(\log t)$ lower bound on the (Euler) $\mathbb{Z}_2$-genus for a more general class of graphs that includes the $t$-Kuratowski graphs of types f), g) and h). 

The definition of gluing a pair of vertices from Subsection~\ref{subsection_zed} can be extended in a straightforward way to \emph{gluing} an arbitrary finite \emph{set of vertices}.
Let $H$ be a $2$-connected graph and let $x,y$ be two nonadjacent vertices of $H$. Let $t$ be a positive integer. The \emph{$2$-amalgamation} of $t$ copies of $H$ (with respect to $x$ and $y$), denoted by $\amalg_{x,y} tH$, is the graph obtained from $t$ disjoint copies of $H$ by gluing all $t$ copies of $x$ into a single vertex and gluing all $t$ copies of $y$ into a single vertex. The two vertices obtained by gluing are again denoted by $x$ and $y$.

An \emph{$xy$-wing} is a $2$-connected graph $H$ with two nonadjacent vertices $x$ and $y$ such that the subgraph $H-x-y$ is connected, and 
the graph obtained from $H$ by adding the edge $xy$
is nonplanar. Clearly, the graphs $K_5 - e$ and $K_{3,3}-e$, where $e=xy$, are $xy$-wings, and similarly $K_{3,3}$, with nonadjacent vertices $x$ and $y$, is an $xy$-wing. The $t$-Kuratowski graphs of types f) and g) are obtained from $\amalg_{x,y} t(K_5 - e)$ and $\amalg_{x,y} t(K_{3,3} - e)$, respectively, by adding the edge $xy$, whereas the $t$-Kuratowski graph of type h) is exactly the $2$-amalgamation $\amalg_{x,y} t(K_{3,3})$. 
See Figure~\ref{obr_wing} for an illustration of $2$-amalgamations of two $xy$-wings.

\begin{figure}
\begin{center}
\includegraphics{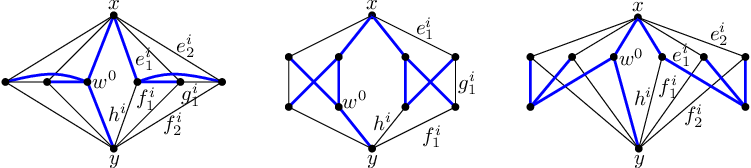}
\end{center}
\caption{$2$-amalgamations of two Kuratowski $xy$-wings. The spanning tree $T$ is drawn bold.}
\label{obr_wing}
\end{figure}

Let $H$ be an $xy$-wing. We will use the following notation. Let $w$ be a vertex of $H$ adjacent to $y$ and let $F'$ be a spanning tree of $H-x-y$.
Let $F$ be a spanning tree of $H-y$ extending $F'$. In the $2$-amalgamation $\amalg_{x,y} tH$ we distinguish the $i$th copy of $H$, its vertices, edges, and subgraphs, by the superscript $i\in \{0,1,\dots, t-1\}$. In particular, for every $i\in \{0,1,\dots, t-1\}$, $H^i$ is an induced subgraph of $\amalg_{x,y} tH$, $F^i$ is a spanning tree of $H^i-y$ and $x$ is a leaf of $F^i$. For a given $t$, let 
\[
T=yw^0 + \bigcup_{i=0}^{t-1}F_i
\]
be a spanning tree of $\amalg_{x,y} tH$. For every edge $e\in E(\amalg_{x,y} tH) \setminus E(T)$, let $C_e$ be the fundamental cycle of $e$ with respect to $T$; that is, the unique cycle in $T+e$.

Enumerate the edges of $E(H)\setminus E(F)$ incident to $x$ as $e_1,\dots,e_k$, the edges of $E(H)\setminus E(F)\setminus \{yw\}$ incident to $y$ as $f_1,\dots,f_l$, and the edges of $E(H-x-y)\setminus E(F)$ as $g_1,\dots,g_m$. Let $h$ be the edge $yw$. Thus, for every $i\in[t-1]$, we have $E(H^i)\setminus E(T)=\{e_1^i, \dots, e_k^i\}\cup\{f_1^i, \dots, f_l^i\}\cup \{g_1^i, \dots, g_m^i\} \cup \{h^i\}$. 

If $C$ and $C'$ are cycles in $\amalg_{x,y} tH$, we denote by $C+C'$ the element of the cycle space of $\amalg_{x,y} tH$ obtained by adding $C$ and $C'$ mod $2$. We also regard $C+C'$ as a subgraph of $\amalg_{x,y} tH$ with no isolated vertices. Note that if $C$ and $C'$ are fundamental cycles sharing at least one edge then $C+C'$ is again a cycle.

\begin{observation}
\label{obs_cykly_mimo_x_y}
Let $i\in [t-1]$. Refer to Figure~\ref{obr_fund_cycles}.
\begin{enumerate}
\item[{\rm a)}] For every $j\in [k]$, the cycle $C_{e^i_j}$ is a subgraph of $H^i-y$.
\item[{\rm b)}] For every $j\in [l]$, the cycle $C_{f^i_j}+C_{h^i}$ is a subgraph of $H^i-x$.
\item[{\rm c)}] For every $j\in [m]$, the cycle $C_{g^i_j}$ is a subgraph of $H^i-x-y$. 
\end{enumerate} 
The cycles $C_{e^i_j}$ with $j\in [k]$, $C_{f^i_j}+C_{h^i}$ with $j\in [l]$, and $C_{g^i_j}$ with $j\in [m]$ generate the cycle space of $H^i$; in particular, they are the fundamental cycles of $H^i$ with respect to the spanning tree $F^i + yw^i$.
\qed
\end{observation}

\begin{figure}
\begin{center}
\includegraphics{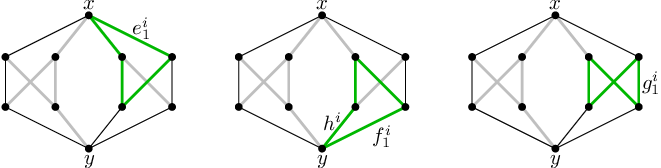}
\end{center}
\caption{Examples of cycles $C_{e^i_1}$, $C_{f^i_1}+C_{h^i}$ and $C_{g^i_1}$ in an amalgamation of Kuratowski $xy$-wings.}
\label{obr_fund_cycles}
\end{figure}

\begin{corollary}
\label{cor_disjunktni_cykly}
Let $i,i'\in [t-1]$ be distinct indices. Then the following pairs of cycles are vertex-disjoint, for all possible pairs of indices $j,j'$:
\begin{enumerate}
\item[{\rm a)}] $C_{e^i_j}$ and $C_{f^{i'}_{j'}}+C_{h^{i'}}$,
\item[{\rm b)}] $C_{f^i_j}+C_{h^i}$ and $C_{g^{i'}_{j'}}$,
\item[{\rm c)}] $C_{e^i_j}$ and $C_{g^{i'}_{j'}}$,
\item[{\rm d)}] $C_{g^{i}_j}$ and $C_{g^{i'}_{j'}}$.
\end{enumerate} 
\end{corollary}

Our first lower bound on the (Euler) $\mathbb{Z}_2$-genus of $2$-amalgamations of $xy$-wings is similar to Proposition~\ref{prop_K3t_log}, and combines the pigeonhole principle and Lemma~\ref{lemma_iocr_cycles}.

\begin{proposition}
Let $H$ be an $xy$-wing. Then $2\mathrm{g}_0(\amalg_{x,y} tH) \ge \mathrm{eg}_0(\amalg_{x,y} tH)\ge \Omega(\log t)$.
\end{proposition}

\begin{proof}
Let $\mathcal{D}$ be an independently even drawing of $\amalg_{x,y} tH$ on a surface $S$ of Euler genus $g$.
For every $i\in[t-1]$ and $e\in E(H)\setminus E(F)$, let $\gamma(e^i)$ be the closed curve representing $C_{e^i}$ in $\mathcal{D}$. 

The homology class $[\gamma(e^i)]$ has one of $2^{g}$ possible values in $H_1(S;\mathbb{Z}_2)$.
Thus, if $t\ge 2^{g(k+l+m+1)} + 2$, then there are distinct indices $i,i'\in[t-1]$ such that for every $e\in E(H)\setminus E(F)$ we have $[\gamma(e^i)]=[\gamma(e^{i'})]$. Using this, we can compute the intersection form for certain pairs of cycles by replacing them with vertex-disjoint pairs; this gives the left equality in each of the following formulas. The right equality follows from Corollary~\ref{cor_disjunktni_cykly} and Observation~\ref{obs_disjoint_cycles}.
In particular, for all possible pairs of indices $j,j'$, we have
\begin{align}
\Omega_{S}([\gamma(e^i_j)],[\gamma(f^i_{j'})]+[\gamma(h^i)])=\Omega_{S}([\gamma(e^i_j)],[\gamma(f^{i'}_{j'})]+[\gamma(h^{i'})])&=0,\label{eq1} \\
\Omega_{S}([\gamma(f^i_{j})]+[\gamma(h^i)],[\gamma(g^i_{j'})])=\Omega_{S}([\gamma(f^i_{j})]+[\gamma(h^i)],[\gamma(g^{i'}_{j'})])&=0,\label{eq2} \\
\Omega_{S}([\gamma(e^i_j)],[\gamma(g^i_{j'})])=\Omega_{S}([\gamma(e^i_j)],[\gamma(g^{i'}_{j'})])&=0,\label{eq3} \\
\Omega_{S}([\gamma(g^i_{j})],[\gamma(g^i_{j'})])=\Omega_{S}([\gamma(g^i_{j})],[\gamma(g^{i'}_{j'})])&=0.\label{eq4}
\end{align}

\begin{figure}
\begin{center}
\includegraphics{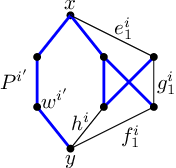}
\end{center}
\caption{An example of the graph $H^{i,i'}$. Its spanning tree $F^{i,i'}$ is drawn bold.}
\label{obr_Hii}
\end{figure}

Let $H^{i,i'}$ be the union of the graph $H^i$ with the unique $xy$-path $P^{i'}$ in $F^{i'}+ yw^{i'}$; see Figure~\ref{obr_Hii}. Since $H$ is an $xy$-wing, the graph $H^{i,i'}$ is nonplanar. The graph $F^{i,i'}=F^i \cup P^{i'}$ is a spanning tree of $H^{i,i'}$, and $E(H^{i,i'})\setminus E(F^{i,i'}) = E(H^i)\setminus E(T)$. 

The fundamental cycle $C'_{h^i}$ of $h^i$ in $H^{i,i'}$ with respect to $F^{i,i'}$ is equal to $C_{h^i}+C_{h^{i'}}$. Since $[\gamma(h^i)]=[\gamma(h^{i'})]$, the cycle $C'_{h^i}$ is homologically zero. 

For every $j\in[k]$, the fundamental cycle of $e^i_j$ in $H^{i,i'}$ with respect to $F^{i,i'}$ is $C_{e^i_j}$ and its homology class in $\mathcal{D}$ is $[\gamma(e^i_j)]$.

For every $j\in[l]$, the fundamental cycle of $f^i_j$ in $H^{i,i'}$ with respect to $F^{i,i'}$ is $C_{f^i_j}+C_{h^{i'}}$ and its homology class is $[\gamma(f^i_{j'})]+[\gamma(h^{i'})]=[\gamma(f^i_{j'})]+[\gamma(h^i)]$.

For every $j\in[m]$, the fundamental cycle of $g^i_j$ in $H^{i,i'}$ with respect to $F^{i,i'}$ is $C_{g^i_j}$ and its homology class in $\mathcal{D}$ is $[\gamma(g^i_j)]$.

By~\eqref{eq1}--\eqref{eq4}, for every pair of independent edges in $E(H^{i,i'})\setminus E(F^{i,i'})$, the homology classes of their fundamental cycles are orthogonal with respect to $\Omega_{S}$. This is a contradiction with Lemma~\ref{lemma_iocr_cycles} applied to $H^{i,i'}$ and the spanning tree $F^{i,i'}$. Therefore, $t\le 2^{g(k+l+m+1)} + 1$.
\end{proof}

To prove the lower bound in Theorem~\ref{theorem_z2}b), we follow the idea of the previous proof and again replace the pigeonhole principle with an argument from linear algebra.
We will also need the following stronger variant of the Hanani--Tutte theorem and Lemma~\ref{lemma_iocr_cycles} for the graphs $K_5$ and $K_{3,3}$. 

\begin{lemma}[{Kleitman~\cite{Kle76_parity}}]
\label{lemma_parity_K5_K3_3}
In every drawing of $K_5$ and $K_{3,3}$ in the plane the total number of pairs of independent edges crossing an odd number of times is odd.
\end{lemma}

Lemma~\ref{lemma_parity_K5_K3_3} was also implicitly proved by Sz\'ekely~\cite[Sections 7 and 8]{Sze04_iocr}.

\begin{corollary}
\label{cor_kuratowski_forest}
Let $G=K_5$ or $G=K_{3,3}$. Let $F$ be a forest in $G$. Let $\mathcal{E}$ be a crosscap drawing of $G$ from Lemma~\ref{lemma_forest}. Then there are an odd number of pairs of independent edges $e,f$ in $E(G)\setminus E(F)$ such that $y_e^{\top}y_f=1$.
\qed
\end{corollary}

The following simple fact is a key ingredient in the proof of Lemma~\ref{lemma_parity_K5_K3_3}.

\begin{observation}
\label{obs_euler}
The graph obtained from each of $K_5$ and $K_{3,3}$ by removing an arbitrary pair of adjacent vertices is a cycle; in particular, all of its vertices have an even degree.
\qed
\end{observation}

An $xy$-wing $H$ is called a \emph{Kuratowski $xy$-wing} if $H$ is one of the graphs $K_5-e$ where $e=xy$, $K_{3,3}-e$ where $e=xy$, or $K_{3,3}$; see Figure~\ref{obr_wing}. Observation~\ref{obs_euler} implies the following important property of Kuratowski $xy$-wings.

\begin{observation}
\label{obs_even_independent_edges_with_xu}
Let $H$ be a Kuratowski $xy$-wing and let $u$ be a vertex adjacent to $x$ in $H$. Then $H-x-u$ is a cycle; in particular, $y$ is incident to exactly two edges in $H-x-u$.
\qed
\end{observation}

In the following key lemma we keep using the notation for the $2$-amalga\-ma\-tion $\amalg_{x,y} tH$ established earlier in this subsection.

\begin{lemma}
\label{lemma_odd_cr_curves}
Let $t\ge 2$, let $H$ be a Kuratowski $xy$-wing and let $\mathcal{D}$ be an independently even drawing of $\amalg_{x,y} tH$ on a surface $S$. Then for every $i\in [0,t-1]$ the graph $H^i$ has two cycles $C_1^i$ and $C_2^i$ such that 
\begin{itemize}
\item ($C_1^i$ is a subgraph of $H^i-x$ and $C_2^i$ is a subgraph of $H^i-y$) or $C_2^i$ is a subgraph of $H^i-x-y$, and
\item the closed curves $\gamma_1^i$ and $\gamma_2^i$ representing $C_1^i$ and $C_2^i$, respectively, in $\mathcal{D}$ satisfy\linebreak $\Omega_{S}([\gamma_1^i],[\gamma_2^i])=1$.
\end{itemize}
\end{lemma}

\begin{proof}
For every $i\in[t-1]$, let $H^{i,0}$ be the union of the graph $H^i$ with the unique $xy$-path $P^{0}$ in $F^{0}+ yw^0$. 
The graph $F^{i,0}=F^i \cup P^{0}$ is a spanning tree of $H^{i,0}$, and $E(H^{i,0})\setminus E(F^{i,0}) = E(H^i)\setminus E(T)$. 

Let $\mathcal{E}$ be a crosscap drawing of $\amalg_{x,y} tH$ from Lemma~\ref{lemma_forest}. 
If $H=K_{3,3}$, we apply Corollary~\ref{cor_kuratowski_forest} to $G=H^i$ and $F=F^i$. If $H=K_5-e$ or $H=K_{3,3}-e$ where $e=xy$, we apply Corollary~\ref{cor_kuratowski_forest} to $G=H^i+e$, $F=F^i+e$, and the drawing of $H^i+e$ where $e$ is drawn along the path $P^0$ in $\mathcal{E}$ (with self-crossings removed if necessary). In each of the three cases, we have an odd number of pairs of independent edges $e,f$ in $E(H_i)\setminus E(T)$ such that $y_e^{\top}y_f=1$.
By Observation~\ref{obs_even_independent_edges_with_xu}, for each $j\in[k]$, there are exactly two edges in $E(H^i)\setminus E(T)$ incident with $y$ and independent from $e^i_j$; see also Figure~\ref{obr_wing}. 
Therefore, considering all possible pairs of independent edges in $E(H_i)\setminus E(T)$, at least one of the following alternatives occurs:
\begin{enumerate}
\item[1)] $y_{h^i}^{\top}y_{g^i_{1}}=1$,
\item[2)] $y_{e^i_j}^{\top}(y_{f^i_{j'}}+y_{h^i})=1$ for some $j\in [k]$ and $j'\in [l]$.
\end{enumerate}
In further arguments, we no longer use the fact that the pairs of edges involved in the scalar products are independent.

To finish the proof of the lemma for $i\in[t-1]$, we use Observation~\ref{obs_cykly_mimo_x_y}.
In particular, 
in case~1) we choose $C_1^i=C_{h^i}$ and $C_2^i=C_{g^i_{1}}$, and in case~2) we choose $C_1^i=C_{f^i_{j'}}+C_{h^i}$ and $C_2^i=C_{e^i_{j}}$.

Finally, by exchanging the roles of $H^1$ and $H^0$ in $\amalg_{x,y} tH$ in the proof, we also obtain cycles $C_1^0$ and $C_2^0$ with the required properties.
\end{proof}

We are now ready to finish the proof of Theorem~\ref{theorem_z2}b).

\begin{proposition}
Let $t\ge 2$ and let $H$ be a Kuratowski $xy$-wing. Then $\mathrm{g}_0(\amalg_{x,y} tH) \ge \lceil t/2\rceil$ and $\mathrm{eg}_0(\amalg_{x,y} tH) \ge t$.
\end{proposition}

\begin{proof}
Let $\mathcal{D}$ be an independently even drawing of $\amalg_{x,y} tH$ on a surface $S$ of Euler genus $g$.
For every $i\in[0,t-1]$, let $C^i_1$ and $C^i_2$ be the cycles from Lemma~\ref{lemma_odd_cr_curves} and let $\gamma^i_1$ and $\gamma^i_2$, respectively, be the closed curves representing them in $\mathcal{D}$. 

Without loss of generality, we assume that there is an $s\in[0,t-1]$ such that 
\begin{itemize}
\item for every $i \in [0,s]$, $C_1^i$ is a subgraph of $H^i-x$ and $C_2^i$ is a subgraph of $H^i-y$, and
\item for every $i \in [s+1,t-1]$, the cycle $C_2^i$ is a subgraph of $H^i-x-y$.
\end{itemize}
It follows that for distinct $i,i'\in [0,t-1]$, the cycles $C^i_1$ and $C^{i'}_2$ are vertex-disjoint whenever $i,i'\in [0,s]$, $i,i' \in [s+1,t-1]$, or $i\le s < i'$.

Let $A$ be the $t\times t$ matrix with entries 
\[
A_{i,i'}=\Omega_{S}([\gamma^i_1],[\gamma^{i'}_2]).
\]
By Lemma~\ref{lemma_odd_cr_curves}, Observation~\ref{obs_disjoint_cycles} and the previous discussion, the matrix $A$ has $1$-entries on the diagonal and $0$-entries above the diagonal; see Figure~\ref{obr_matice}. Thus, the rank of $A$ over $\mathbb{Z}_2$ is $t$. Hence, the rank of $\Omega_{S}$ is at least $t$, which implies $g\ge t$ using Fact~\ref{fact_full_rank}. 
\end{proof}

\begin{figure}
\[
\begin{pmatrix}
1 & 0 & 0 & 0 \\
0 & 1 & 0 & 0 \\
* & * & 1 & 0 \\
* & * & 0 & 1 \\
\end{pmatrix}
\]
\caption{An example of the matrix $A$ with $t=4$ and $s=1$. The entries marked with $*$ may be equal to $0$ or $1$; the remaining entries are determined uniquely.}
\label{obr_matice}
\end{figure}

\section{Acknowledgements}
We thank Zden\v{e}k Dvo\v{r}\'ak, Xavier Goaoc and Pavel Pat\'ak for helpful discussions. 
We also thank Bojan Mohar, Paul Seymour, Gelasio Salazar, Jim Geelen and John Maharry for information about their unpublished results related to Conjecture~\ref{conjecture_folklore}. Finally we thank the reviewers for corrections and suggestions for improving the presentation.




\begin{thebibliography}{10}

\bibitem{BHKY62_additivity}
J. Battle, F. Harary, Y. Kodama and J. W. T. Youngs, 
Additivity of the genus of a graph,
{\em Bull. Amer. Math. Soc.} {\bf 68} (1962), 565--568. 

\bibitem{BKMM_7connected}
T. B{\"o}hme, K. Kawarabayashi, J. Maharry and B. Mohar,
$K_{3,k}$-minors in large $7$-connected graphs, preprint available at \url{http://preprinti.imfm.si/PDF/01051.pdf} (2008).

\bibitem{BKMM_linearconnect}
T. B{\"o}hme, K. Kawarabayashi, J. Maharry and B. Mohar,
Linear connectivity forces large complete bipartite minors,
{\em J. Combin. Theory Ser. B\/} {\bf 99}(3) (2009), 557--582.

\bibitem{Bo78_Kmn}
A. Bouchet,
Orientable and nonorientable genus of the complete bipartite graph,
{\em J. Combin. Theory Ser. B\/} {\bf 24}(1) (1978), 24--33.

\bibitem{Ca91_tournament}
D. de Caen,
The ranks of tournament matrices,
{\em Amer. Math. Monthly\/} {\bf 98}(9) (1991), 829--831. 

\bibitem{CN00_thrackles}
G.~Cairns and Y.~Nikolayevsky, Bounds for generalized thrackles, 
{\em Discrete Comput. Geom.} {\bf 23}(2) (2000), 191--206.

\bibitem{CRS15_continuum}
R. Christian, R. B. Richter and G. Salazar,
Embedding a graph-like continuum in some surface,
{\em J. Graph Theory\/} {\bf 79}(2) (2015), 159--165. 

\bibitem{Code17_computational}
{\'E}. Colin de Verdi{\`e}re,
Computational topology of graphs on surfaces,
{\em Handbook of discrete and computational geometry}, Third edition, Edited by Jacob E. Goodman, Joseph O'Rourke and Csaba D. T\'oth, Discrete Mathematics and its Applications (Boca Raton), CRC Press, Boca Raton, FL, 2018. ISBN: 978-1-4987-1139. Electronic version: \url{http://www.csun.edu/~ctoth/Handbook/HDCG3.html} (accessed January 2022).

\bibitem{CKPPT16_direct}
{\'E}. Colin de Verdi{\`e}re, V. Kalu{\v{z}}a, P. Pat{\'a}k, Z. Pat{\'a}kov{\'a} and M. Tancer,
A direct proof of the strong Hanani--Tutte theorem on the projective plane,
{\em J. Graph Algorithms Appl.} {\bf 21}(5) (2017), 939--981. 

\bibitem{DGH85_2amalgamation}
R. W. Decker, H. H. Glover and J. P. Huneke,
Computing the genus of the 2-amalgamations of graphs,
{\em Combinatorica\/} {\bf 5}(4) (1985), 271--282. 

\bibitem{Di16_fifth}
R.~Diestel, 
{\em Graph Theory}, Fifth edition, 
{\em Graduate Texts in Mathematics\/} 173, Springer, 2017. ISBN: 978-3-662-53621-6.

\bibitem{FHRR95_projective}
J. R. Fiedler, J. P. Huneke, R. B. Richter and N. Robertson, 
Computing the orientable genus of projective graphs,
{\em J. Graph Theory\/} {\bf 20}(3) (1995), 297--308. 

\bibitem{FrMu11_alternative}
J.-O. Fr{\"o}hlich and T. M{\"u}ller, 
Linear connectivity forces large complete bipartite minors: an alternative approach,
{\em J. Combin. Theory Ser. B\/} {\bf 101}(6) (2011), 502--508.

\bibitem{FK17_genus4}
R. Fulek and J. Kyn\v{c}l,
Counterexample to an extension of the Hanani--Tutte theorem on the surface of genus 4,
{\em Combinatorica\/} {\bf 39} (2019), Issue 6, 1267--1279.

\bibitem{GrTu01_theory}
J. L. Gross,  and T. W. Tucker, 
{\em Topological graph theory},
Dover Publications, Inc., Mineola, NY, 2001. ISBN: 0-486-41741-7.

\bibitem{Ha34_uber}
H.~Hanani, 
{\"{U}}ber wesentlich unpl{\"{a}}ttbare {K}urven im drei-dimensionalen {R}aume, 
{\em Fundamenta Mathematicae\/} {\bf 23} (1934), 135--142.

\bibitem{Hat02_AT}
A. Hatcher, 
{\em Algebraic topology}, Cambridge University Press, Cambridge, 2002. ISBN: 0-521-79160-X. Electronic version: \url{http://pi.math.cornell.edu/~hatcher/AT/ATpage.html} (accessed January 2019).

\bibitem{Hau14_mod2}
J.-C. Hausmann, 
{\em Mod two homology and cohomology},
Universitext, Springer, Cham, 2014. ISBN: 978-3-319-09353-6; 978-3-319-09354-3.

\bibitem{KaMo07_survey}
K. Kawarabayashi and B. Mohar,
Some recent progress and applications in graph minor theory,
{\em Graphs Combin.} {\bf 23}(1) (2007), 1--46.

\bibitem{Kle76_parity}
D. J. Kleitman,
A note on the parity of the number of crossings of a graph,
{\em J. Combinatorial Theory Ser. B\/} {\bf 21}(1) (1976), 88--89. 

\bibitem{Ky20_issue}
J. Kyn\v{c}l, Issue UPDATE: in graph theory, different definitions of edge crossing numbers - impact on applications?, MathOverflow, Answer to a question of user161819, \url{https://mathoverflow.net/a/366876} (2020).

\bibitem{LoMa11_arf}
M. Loebl and G. Masbaum,
On the optimality of the Arf invariant formula for graph polynomials,
{\em Adv. Math.} {\bf 226}(1) (2011), 332--349. 

\bibitem{MN09_invitation}
J. Matou\v{s}ek and J. Ne\v{s}et\v{r}il,
{\em Invitation to Discrete Mathematics}, Second edition, 
Oxford University Press, Oxford, 2009. ISBN: 978-0-19-857042-4.

\bibitem{Mi87_additivity}
G. L. Miller, 
An additivity theorem for the genus of a graph,
{\em J. Combin. Theory Ser. B\/} {\bf 43}(1) (1987), 25--47. 

\bibitem{Mo01_survey}
B. Mohar,
Graph minors and graphs on surfaces, 
{\em Surveys in combinatorics, 2001 (Sussex)}, 145--163,
{\em London Math. Soc. Lecture Note Ser.} 288, Cambridge Univ. Press, Cambridge, 2001. 

\bibitem{MT01_graphs}
B.~Mohar and C.~Thomassen, 
{\em Graphs on surfaces\/}, 
Johns Hopkins Studies in the Mathematical Sciences, Johns Hopkins University Press, Baltimore, MD, 2001. ISBN 0-8018-6689-8.

\bibitem{PSS09_pp}
M.~J. Pelsmajer, M.~Schaefer and D. Stasi, 
Strong Hanani--Tutte on the projective plane,
\emph{SIAM J. Discrete Math.} {\bf 23}(3) (2009), 1317--1323. 

\bibitem{PSS09_surfaces}
M.~J. Pelsmajer, M.~Schaefer and D.~{\v{S}}tefankovi{\v{c}},
Removing even crossings on surfaces,
{\em European J. Combin.} {\bf 30}(7) (2009), 1704--1717. 

\bibitem{R87_2amalgamation}
R. B. Richter,
On the Euler genus of a 2-connected graph,
{\em J. Combin. Theory Ser. B} {\bf 43}(1) (1987), 60--69. 

\bibitem{Ri65_Kmn}
G. Ringel,
Das Geschlecht des vollst\"{a}ndigen paaren {G}raphen,
{\em Abh. Math. Sem. Univ. Hamburg} {\bf 28} (1965), 139--150. 

\bibitem{Ri65b_Kmn_Euler}
G. Ringel,
Der vollst\"{a}ndige paare {G}raph auf nichtorientierbaren {F}l\"{a}chen,
{\em J. Reine Angew. Math.} {\bf 220} (1965), 88--93.

\bibitem{RoSe90_VII_paths}
N. Robertson and P. D. Seymour, 
Graph minors. VII. Disjoint paths on a surface,
{\em J. Combin. Theory Ser. B\/} {\bf 45}(2) (1988), 212--254. 

\bibitem{RoVi90_representativity}
N. Robertson and R. Vitray,
Representativity of surface embeddings,
{\em Paths, flows, and VLSI-layout (Bonn, 1988)}, 293--328,
{\em Algorithms Combin.} 9, Springer, Berlin, 1990.

\bibitem{Sch13_hananitutte}
M.~Schaefer, 
Hanani-{T}utte and related results, {\em Geometry---intuitive, discrete, and convex}, vol.~24 of {\em Bolyai Soc. Math. Stud.}, 259--299, J\'anos
Bolyai Math. Soc., Budapest (2013).

\bibitem{SS13_block}
M. Schaefer and D.~{\v{S}}tefankovi{\v{c}},
Block additivity of $\mathbb{Z}_2$-embeddings,
{\em Graph drawing, Lecture Notes in Computer Science\/} 8242, 185--195, Springer, Cham, 2013.

\bibitem{Seym17}
Paul Seymour, personal communication, 2017.

\bibitem{StBe77_blocks}
S. Stahl and L. W. Beineke,
Blocks and the nonorientable genus of graphs,
{\em J. Graph Theory} {\bf 1}(1) (1977), 75--78. 

\bibitem{Sze04_iocr}
L. A. Sz{\'e}kely, 
A successful concept for measuring non-planarity of graphs: the crossing number,
{\em Discrete Math.} {\bf 276}(1-3) (2004), 331--352.

\bibitem{Tutte70_toward}
W.~T. Tutte, Toward a theory of crossing numbers, 
{\em J. Combinatorial Theory\/} {\bf 8} (1970), 45--53.

\end{thebibliography}
\end{document}